\documentclass[11pt]{amsart}
\usepackage{amssymb}
\usepackage{color}
\usepackage{enumerate}
\usepackage[pdftex]{graphicx}
\newtheorem{thm}{Theorem}[section]

\newtheorem{lemma}[thm]{Lemma}

\newtheorem{notation}[thm]{Notation}

\numberwithin{equation}{section}

\newtheorem{quest}[thm]{Question}

\def\gcd{{\rm gcd}}

\def\cent#1#2{{\bf C}_{#1}(#2)}

\def\zent#1{{\bf Z}(#1)}

\def\Sm0#1{{{\rm GL}}(1,#1)}

\begin{document}

\title[Complete bipartite divisor graphs of finite groups]
{Groups having complete bipartite divisor graphs for their conjugacy class sizes}

\author[R.~Hafezieh]{Roghayeh Hafezieh}
\address{R.~Hafezieh, Department of Mathematics \newline Gebze Institute of Technology
\\ Gebze, Turkey}
\email{roghayeh@gyte.edu.tr}

\author[P.~Spiga]{Pablo Spiga}
\address{P.~Spiga, Dipartimento di Matematica e Applicazioni,
\newline University of Milano-Bicocca
\\Via Cozzi 53 Milano, MI 20125, Italy}
\email{pablo.spiga@unimib.it}

\begin{abstract}
Given a finite group $G$, the \textit{bipartite divisor graph} for its conjugacy class sizes is the bipartite graph with bipartition consisting of the set of conjugacy class sizes of $G\setminus\zent G$ (where $\zent G$ denotes the centre of $G$) and the set of prime numbers that divide these conjugacy class sizes, and with $\{p,n\}$  being an edge if $\gcd(p,n)\neq 1$.

 In this paper we construct infinitely many groups whose bipartite divisor graph for their conjugacy class sizes is the complete bipartite graph $K_{2,5}$, giving a solution to a question of Taeri~\cite{8T}.
\end{abstract}

\subjclass[2010]{Primary 00Z99; secondary 99A00}
\keywords{bipartite divisor graph, conjugacy class size, extra-special group}

\maketitle
\section{Introduction}
\label{sec:introd}
Given a finite group $G$, there are several sets of invariants that convey nontrivial information about the structure of $G$. Some classical examples include the set consisting of the orders of the elements of \(G\),  or the set of conjugacy class sizes of \(G\), or the set of degrees of the irreducible complex characters of \(G\).

For every set $X$ as above, it is  natural to ask to what extent the group structure of \(G\) is reflected and influenced by $X$, and a useful tool in this kind of investigation is the so-called \emph{bipartite divisor graph} $B(X)$ of $X$. The bipartition of the vertex set of $B(X)$ consists of $X\setminus\{1\}$ and of the set of prime numbers dividing $x$, for some $x\in X$, and the edge set of $B(X)$ consists of the pairs $\{p,x\}$ with $\gcd(p,x)\neq 1$.

 In this paper we are interested in the case that $X$ consists of the conjugacy class sizes of $G$, that is, $X=\{|g^G|\mid g\in G\}$, and hence we study the bipartite divisor graph for the conjugacy class sizes. We will denote this graph simply by $B(G)$. We refer the reader to the beautiful survey~\cite{CC} for
the influence of the conjugacy class sizes on the structure of
finite groups.

Historically, various graphs associated with the algebraic structure of a  finite group have been extensively studied by a large number of authors, see for example~\cite{GA,AIPS,BHM,BDIP,CD,DPSS2,K,Wil}. Recently, Lewis~\cite{L} discussed many remarkable connections among these graphs by analysing analogous of these graphs for arbitrary positive integer
subsets. Then, inspired by the survey of Lewis, Praeger and Iranmanesh~\cite{IP} introduced the bipartite divisor graph $B(X)$ for a finite set
$X$ of positive integers and studied some basic  invariants of this graph
(such as the diameter, girth, number of connected components and
clique number).

One of the main questions that naturally arises in this area is classifying the groups whose bipartite divisor graphs have special graphical shapes. For instance, in~\cite{2HIr}, the first author of this paper and Iranmanesh have classified the groups  whose bipartite  divisor graphs are paths. Similarly, Taeri~\cite{8T} considered the case that the bipartite  divisor graph  is a cycle, and in the course of his investigation posed the following question:

\smallskip

\noindent\textbf{Question. }({{\cite[Question~$1$]{8T}}}) \textit{Is there any finite group $G$ such that $B(G)$ is isomorphic to a complete bipartite graph $K_{m,n}$, for some positive integers $m,n\geq 2$?}

\smallskip

The main theorem of this paper gives a family of  examples which shows that the answer to this question  is positive.

\begin{thm}\label{thm:2.2}
There exist infinitely many groups $G$ with $B(G)=K_{2,5}$.
\end{thm}

In light of Theorem~\ref{thm:2.2}, we pose the following question.
\begin{quest}For which positive integers $m,n\geq 2$, are there infinitely many finite groups with $B(G)=K_{m,n}$?
\end{quest}

In this paper we have tried to produce infinitely many groups $G$ with $B(G)=K_{m,n}$, $m,n\geq 2$ and with $m,n$ as small as possible. We were unable to construct groups $G$ with $B(G)=K_{2,3}$ or $B(G)=K_{2,4}$. Observe that Dolfi and Jabara~\cite{DolfiJabara} have classified  the finite groups having only two non-trivial conjugacy class sizes, and from this remarkable work it follows~\cite[Corollary~C]{DolfiJabara} that there is no group $G$ with $B(G)=K_{2,2}$.



\section{Groups with complete bipartite  divisor graphs for their conjugacy class sizes}

Given a finite group $G$, we denote by $\zent G$ the centre of $G$, and given $g\in G$, we denote by $g^G$ the conjugacy class of $g$ under $G$. Furthermore, $|g^G|$ denotes the cardinality of $g^G$.




\begin{notation}\label{notation}{\rm
We let $p$ be an odd prime number and $E$ be the extra-special $2$-group of  plus-type
\begin{eqnarray*}
E=\langle x_1,\ldots,x_{p-1},y_1,\ldots,y_{p-1},z&\mid&x_i^2=y_i^2=z^2=[x_i,z]=[y_i,z]=1 \,\,\forall i,\\
&&[x_i,x_j]=[y_i,y_j]=[x_i,y_j]=1\, \forall i\neq j,\\
&& [x_i,y_i]=z\,\, \forall i\rangle.
\end{eqnarray*}
Observe that
$$[x_1^{\varepsilon_1}\cdots x_{p-1}^{\varepsilon_{p-1}},y_1^{\eta_1}\cdots y_{p-1}^{\eta_{p-1}}]=z^{\varepsilon_1\eta_1+\cdots+\varepsilon_{p-1}\eta_{p-1}},$$
and  that every element $g\in E$ can be  written  uniquely as
$$g=x_1^{\varepsilon_1}\cdots x_{p-1}^{\varepsilon_{p-1}}y_1^{\eta_1}\cdots y_{p-1}^{\eta_{p-1}}z^\nu$$
 with $\varepsilon_i,\eta_i\in \{0,1\}$ for each $i\in \{1,\ldots,p-1\}$ and $\nu\in\{0,1\}$.

Finally, we let $A$ and $B$ be the maps from $\{x_1,\ldots,x_{p-1},y_1,\ldots,y_{p-1},z\}$ to $E$ given (with exponential notation) by
\begin{align*}
&z^A=z,&&z^B=z,&\\
&x_i^A=x_{i+1},&& x_i^B=x_{p-i},&\textrm{ for each }i\in \{1,\ldots,p-2\},\\
&x_{p-1}^A=x_1\cdots x_{p-1},&&x_{p-1}^B=x_1,&\\
&y_i^A=y_1y_{i+1},&&y_i^B=y_{p-i},& \textrm{ for each }i\in \{1,\ldots,p-2\},\\
&y_{p-1}^A=y_1,&&y_{p-1}^B=y_1.&
\end{align*}
}
\end{notation}
\begin{lemma}\label{lem:2.1}
Let $E$, $A$ and $B$ be as in Notation $\ref{notation}$. Then $A$ and $B$ extend to two automorphisms $a$ and $b$ (respectively) of $E$. Moreover, $a^p=b^2=(ab)^2=1$ and $\langle a,b\rangle$ is a dihedral group of order $2p$.
\end{lemma}
\begin{proof}
To show that $A$ and $B$ extend to two automorphisms, say $a$ and $b$ respectively, of $E$ it suffices to prove that they preserve the defining relations of $E$. For instance, for every $i\in \{1,\ldots,p-2\}$, we have
$[x_i^A,y_i^A]=[x_{i+1},y_1y_{i+1}]=[x_{i+1},y_{i+1}]=z=z^A=[x_i,y_i]^A$. Similarly, $[x_{p-1}^A,y_{p-1}^A]=[x_1\cdots x_{p-1},y_1]=[x_1,y_1]=z=z^A=[x_{p-1},y_{p-1}]^A$. For $i,j\in \{1,\ldots,p-2\}$ with $i\neq j$, we have $[x_i^A,y_j^A]=[x_{i+1},y_1y_{j+1}]=1=[x_i,x_j]^A$. Moreover, $[x_i^A,y_{p-1}^A]=[x_{i+1},y_1]=1=[x_i,y_{p-1}]^A$ and $[x_{p-1}^A,y_j^A]=[x_1\cdots x_{p-1},y_1y_{j+1}]=[x_1,y_1][x_{j+1},y_{j+1}]=z\cdot z=1=[x_{p-1},y_j]^A$.  All the other computations are similar and are left to the conscientious reader.

It is readily seen that $b^2$ fixes each  generator $x_1,\ldots,x_{p-1}$, $y_1,\ldots,y_{p-1}$, $z$ of $E$, and hence $b^2=1$. Observe that $x_{p-1}^{a^2}=(x_1\cdots x_{p-2}x_{p-1})^a=x_1^a\cdots x_{p-2}^ax_{p-1}^a=x_2\cdots x_{p-1}(x_1\cdots x_{p-1})=x_1$. Thus $x_1^{a^p}=(x_1^{a^{p-2}})^{a^2}=x_{p-1}^{a^2}=x_1$, and hence $x_i^{a^p}=(x_1^{a^{i-1}})^{a^p}=(x_1^{a^p})^{a^{i-1}}=x_1^{a^{i-1}}=x_i$ for every $i\in\{1,\ldots,p-1\}$. Arguing inductively on $i$, we get $y_i^{a^{p-1-i}}=y_{p-1-i}y_{p-1}$ for each $i\in \{1,\ldots,p-1\}$. It follows that $y_i^{a^{p-i}}=(y_{p-i-1}y_{p-1})^a=y_1y_{p-i}y_1=y_{p-i}$. Given $i\in \{1,\ldots,p-1\}$, applying this equality first to the index  $i$ and then to the index $p-i$, we get $y_i^{a^p}=(y_i^{a^{p-i}})^{a^i}=y_{p-i}^{a^i}=y_{p-(p-i)}=y_i$. Therefore $a^p$ fixes $x_1,\ldots,x_{p-1},y_1,\ldots,y_{p-1},z$ and hence $a^p=1$. Finally, we have
\begin{align*}
&z^{ab}=z,&\\
&x_i^{ab}=x_{i+1}^b=x_{p-i-1},&\forall i\in \{1,\ldots,p-2\},\\
&x_{p-1}^{ab}=(x_1\cdots x_{p-1})^b=x_{p-1}\cdots x_1=x_1\ldots x_{p-1},&\\
&y_i^{ab}=(y_1y_{i+1})^b=y_{p-1}y_{p-i-1},&\forall i\in\{1,\ldots,p-2\},\\
&y_{p-1}^{ab}=y_1^{b}=y_{p-1},
\end{align*}
from which it easily follows that $(ab)^2=1$.
\end{proof}	

\begin{proof}[Proof of Theorem~$\ref{thm:2.2}$]
Let $p,E,A$ and $B$ be as in Notation~\ref{notation}, let $a$ and $b$ be the automorphisms extending $A$ and $B$ as in Lemma~\ref{lem:2.1}, and  set $M = \langle a,b\rangle$. Given two elements $n_1$ and $n_2$ of order $p$, define $N=\langle n_1\rangle\times\langle n_2\rangle\times E$ and $G=N\rtimes M$ where
$$n_1^a=n_1,\,\, n_2^a=n_1n_2,\,\, n_1^b=n_1,\,\, n_2^b=n_2^{-1}$$
(note that this is well-defined because the action of $M =\langle a,b\rangle$ on $\langle n_1,n_2\rangle$ determines a dihedral group of automorphisms of $\langle n_1,n_2\rangle$ order $2p$).  We show that $B(G)=K_{2,5}$.

We start by determining the conjugacy class sizes of $G\setminus\zent G$. Clearly, $\zent G=\langle n_1,z\rangle$.

We claim that the conjugacy classes of $G$ in $N\setminus \zent G$ have cardinality $2p$ or $4p$. Let $e\in \{h\in E\mid [a,h]\in \zent E\}$ and write $e=x_1^{\varepsilon_1}\cdots x_{p-1}^{\varepsilon_{p-1}}y_1^{\eta_1}\cdots y_{p-1}^{\eta_{p-1}}z^{\nu}$, for some $\varepsilon_1,\ldots,\varepsilon_{p-1},\eta_1,\ldots,\eta_{p-1},\nu\in \{0,1\}$. We have
\begin{eqnarray*}
e^a&=&(x_1^{\varepsilon_1})^a\cdots (x_{p-1}^{\varepsilon_{p-1}})^a (y_1^{\eta_1})^a\cdots (y_{p-1}^{\eta_{p-1}})^a(z^\nu)^a\\
&=&x_2^{\varepsilon_1}x_3^{\varepsilon_2}\cdots x_{p-1}^{\varepsilon_{p-2}}(x_1\cdots x_{p-1})^{\varepsilon_{p-1}}(y_1y_2)^{\eta_1}(y_1y_3)^{\eta_2}\cdots (y_1y_{p-1})^{\eta_{p-2}}y_{1}^{\eta_{p-1}}z^\nu\\
&=&x_1^{\varepsilon_{p-1}}x_2^{\varepsilon_1+\varepsilon_{p-1}}\cdots x_{p-1}^{\varepsilon_{p-2}+\varepsilon_{p-1}}y_1^{\eta_1+\cdots+\eta_{p-1}}y_2^{\eta_1}\cdots y_{p-1}^{\eta_{p-2}}z^\nu.
\end{eqnarray*}
As $e^{-1}e^a=[e,a]\in \zent E=\langle z\rangle$, by comparing $e$ with $e^a$, we get $\varepsilon_i=\eta_i=0$ for every $i\in \{1,\ldots,p-1\}$, and hence $e\in \zent E$. Thus $\{h\in E\mid [a,h]\in \zent E\}=\zent E$. In particular, $\cent E a=\zent E$ and $a$ acts fixed point freely on $E\setminus\zent E$. It follows that  the orbits of $M$ acting by conjugation on $E\setminus \zent E$ have size $p$ or $2p$. Since $|E^{\prime}|=2$, for every $e\in E\setminus\zent E$ we have $|E:\cent E e|=2$. Thus $e^{G}$ has size $2p$ or $4p$. Now let $g\in N\setminus \zent N$. Write $g = ne$, for some $n\in \langle n_{1},n_{2}\rangle\setminus\langle n_{1}\rangle$ and $e\in E\setminus\langle z\rangle$. Now $\cent N g=\langle n_{1},n_{2}\rangle \cent E e$ has index $2$ in $N$ and so (as $a$ does not centralise $n$) $\cent G g$
has index divisible by $2p$. It follows (with an elementary computation) that the conjugacy classes of $G$ in $N\setminus \zent G$ have size $2p$ or $4p$. Observe that both possibilities can occur: $|x_1^G|=4p$ and $|(x_1x_{p-1})^G|=2p$.

Let $g\in \langle N,a\rangle\setminus N$ and write $g=na^i$, for some $n\in N$ and $i\in \{1,\ldots,p-1\}$. Now, given $e\in E$, we have $e^n=ez^\nu$, for some $\nu\in \{0,1\}$. In particular, if $e\in \cent E g$, then $e=e^{g}=(e^n)^{a^i}=(ez^\nu)^{a^i}=e^{a^i}z^\nu$, and hence $[a,e]\in \zent E$. Thus $e\in \zent E$. It follows that  $\cent N g=\zent G$ and $|\cent G g|=p|\cent N g |=2p^2$. Thus $|g^G|=2^{2p-1}p$.

It remains to compute the size of the conjugacy classes of $G$ in $G\setminus \langle a,N\rangle$. Observe that every element of $G\setminus\langle a,N\rangle$ is conjugate to an element of $\langle b,N\rangle\setminus N$. In particular, since we are  interested only on the conjugacy class sizes, we may assume that $g\in \langle b,N\rangle\setminus N$. Write $g=nb$, for some $n\in N$. Observe that
\begin{eqnarray*}
\{h\in E\mid [h,b]\in \zent E\}&=&\langle x_1x_{p-1},x_2x_{p-2},\ldots,x_{(p-1)/2}x_{(p+1)/2},\\
&&y_1y_{p-1},y_2y_{p-2},\ldots,y_{(p-1)/2}y_{(p+1)/2},z\rangle=\cent E b,
\end{eqnarray*}
and $|\cent E b|=2^p$.
Now, arguing as above, it is easy to verify that $\cent N g\leq \langle n_1,\cent E b\rangle$ and $|\langle n_1,\cent E b\rangle:\cent N g|\in \{1,2\}$. Thus $|\cent G g|\in \{p^22^{p},p^22^{p+1}\}$ and hence $|g^G|$ equals $2^{p-1}p^2$ or $2^pp^2$. Observe that both of these cases can occur: in fact $|b^G|=|G:\cent G b|=2^{p-1}p^2$ and $|(x_1b)^G|=|G:\cent G {x_1b}|=2^pp^2$.

Summing up, the conjugacy class sizes $|g^G|$ (with $g\in G\setminus \zent G$) are $2p,4p,2^{2p-1}p,2^{p-1}p^2,2^pp^2$, from which it follows that $B(G)=K_{2,5}$.
\end{proof}

\end{document}